\documentclass[11pt]{amsart}

\usepackage[a4paper,hmargin=3.5cm,vmargin=3.9cm]{geometry}
\usepackage{amsfonts,amssymb,amscd,amstext}
\usepackage{graphicx}
\usepackage[dvips]{epsfig}
\usepackage{quoting}
\usepackage{hyperref}

\usepackage{fancyhdr}
\input xy
\xyoption{all}

\usepackage{enumerate}
\usepackage{titlesec}
\usepackage{mathrsfs}

\titleformat{\section}
{\filcenter\bfseries} {\thesection{.}}{0.2cm}{}

\titleformat{\subsection}[runin]
{\bfseries} {\thesubsection{.}}{0.15cm}{}[.]

\titleformat{\subsubsection}[runin]
{\em}{\thesubsubsection{.}}{0.15cm}{}[.]

\usepackage[up,bf]{caption}

\newtheorem{theorem}{Theorem}[section]
\newtheorem{proposition}[theorem]{Proposition}
\newtheorem{lemma}[theorem]{Lemma}
\newtheorem{claim}[theorem]{Claim}
\newtheorem{corollary}[theorem]{Corollary}

\theoremstyle{definition}
\newtheorem{definition}[theorem]{Definition}
\newtheorem{remark}[theorem]{Remark}

\newtheorem{example}[theorem]{Example}
\newtheorem{question}[theorem]{Question}

\numberwithin{equation}{section}
\numberwithin{figure}{section}

\newcommand\Cscr{\mathscr{C}}

\newcommand\Pscr{\mathscr{P}}

\newcommand\C{\mathbb{C}}

\renewcommand\P{\mathbb{P}}

\newcommand\R{\mathbb{R}}

\newcommand\Z{\mathbb{Z}}

\renewcommand\c{\mathbb{C}}

\newcommand\igot{\mathfrak{i}}

\renewcommand\igot{\mathfrak{i}}

\renewcommand\imath{\igot}

\usepackage[color=blue!20]{todonotes}

\usepackage{color}

\begin{document}


\fancyhead[LO]{Regular immersions directed by algebraically elliptic cones}
\fancyhead[RE]{A.\ Alarc\'on and F.\ L\'arusson}
\fancyhead[RO,LE]{\thepage}

\thispagestyle{empty}



\begin{center}
{\bf \Large Regular immersions
\\ \vspace{1mm}  
directed by algebraically elliptic cones} 

\bigskip

%
%
{\bf Antonio Alarc\'on \; and \; Finnur L\'arusson}
\end{center}


%
%

\bigskip

\begin{quoting}[leftmargin={5mm}]
{\small
\noindent {\bf Abstract}\hspace*{0.1cm}  
Let $M$ be an open Riemann surface and $A$ be the punctured cone in $\C^n\setminus\{0\}$ on a smooth projective variety $Y$ in $\P^{n-1}$.  Recently, Runge approximation theorems with interpolation for holomorphic immersions $M\to\C^n$, directed by $A$, have been proved under the assumption that $A$ is an Oka manifold.  We prove analogous results in the algebraic setting, for regular immersions directed by $A$ from a smooth affine curve $M$ into $\C^n$.  The Oka property is naturally replaced by the stronger assumption that $A$ is algebraically elliptic, which it is if $Y$ is uniformly rational.  Under this assumption, a homotopy-theoretic necessary and sufficient condition for approximation and interpolation emerges.  We show that this condition is satisfied in many cases of interest.
 

\noindent{\bf Keywords}\hspace*{0.1cm} 
Riemann surface, holomorphic immersion, directed immersion, regular immersion, algebraically elliptic manifold, Oka manifold, Runge approximation, minimal surface


\noindent{\bf Mathematics Subject Classification (2020)}\hspace*{0.1cm} 
Primary 32E30. Secondary 30F99, 32H02, 32Q56, 53A10, 53C42

\noindent{\bf Date}\hspace*{0.1cm} 
5 December 2023.  Minor edits 4 November 2024
}

\end{quoting}



\section{Introduction and main results}
\label{sec:intro}

\noindent
Let $A_0$ be a conical closed complex subvariety of $\C^n$, $n\ge 2$.  By Chow's theorem, $A_0$ is algebraic and is the common zero set of finitely many homogeneous holomorphic polynomials.  Assume that $A=A_0\setminus\{0\}\subset\C^n_* = \C^n\setminus\{0\}$ is smooth and connected, so $A$ is the punctured cone on a connected submanifold $Y$ of $\P^{n-1}$.  If $M$ is an open Riemann surface, a holomorphic immersion $M\to\C^n$ is said to be {\em directed by} $A$, or to be an {\em $A$-immersion}, if its complex derivative with respect to any local holomorphic coordinate on $M$ takes its values in $A$.  A systematic investigation of $A$-immersions was made by Alarc\'on and Forstneri\v c in \cite{AlarconForstneric2014IM}.  In particular, under the assumption that $A$ is an Oka manifold not contained in any hyperplane in $\C^n$,  they proved an Oka principle for $A$-immersions, including the Runge and Mergelyan approximation properties.  (For the theory of Oka manifolds, see the monograph \cite{Forstneric2017E}.  We note that $A$ is Oka if and only if $Y$ is.)  In the subsequent paper \cite{AlarconCastro-Infantes2019APDE} by Alarc\'on and Castro-Infantes, interpolation was added to the picture.  A parametric Oka principle for directed immersions was proved by Forstneri\v c and L\'arusson in \cite{ForstnericLarusson2019CAG}.

In this paper, we pursue the algebraic analogues of the above results for directed immersions. We prove a general uniform approximation theorem with
interpolation for $A$-immersions of open Riemann surfaces in affine spaces in the algebraic category.  To this end, throughout the paper, it is natural to assume that the punctured cone $A$ is not only Oka but algebraically elliptic in the sense of Gromov (see \cite{Gromov1989} or \cite[Definition 5.6.13]{Forstneric2017E}).  
Likewise, in the algebraic setting, the open Riemann surface $M$ will be a smooth affine curve, so $M$ is the complement in a compact Riemann surface $\overline M$ of finitely many points $p_1, \ldots, p_m$, which we call the {\em ends\,} of $M$.  We assume, without loss of generality, that $M$ and $\overline M$ are connected.  A complex-valued function or a complex form on $M$ is regular if it is holomorphic and extends meromorphically to  $\overline M$.  
Algebraic Oka theory is more rigid than standard Oka theory.  Obstructions arise that are not present in the holomorphic theory (see \cite{LarussonTruong2019}), so it is no surprise that approximation by regular $A$-immersions introduces new difficulties.
For instance, it is fundamental in the proofs in \cite{AlarconForstneric2014IM,AlarconCastro-Infantes2019APDE,ForstnericLarusson2019CAG} that every open Riemann surface admits a holomorphic 1-form without zeros, while in general a smooth affine curve does not carry any nowhere-vanishing regular 1-forms.
Our method of proof, dealing directly with 1-forms instead of passing to maps, overcomes this complication without being technically much more difficult.

The following is our first main result.  Recall that a regular map $M\to\C^n$ is proper if and only if its meromorphic extension to $\overline M$ has an effective pole (an actual pole of positive order) at each end of $M$.  Also note that $A$ defines a subbundle $\mathcal A$ of $(T^*M)^{\oplus n}$ with fibre isomorphic to $A$, whose sections are $n$-tuples $(\alpha_1,\ldots,\alpha_n)$ of $(1,0)$-forms with no common zeros, such that the ratio $(\alpha_1:\cdots:\alpha_n)$ takes values in $Y$.
%
%
\begin{theorem}   \label{t:main-theorem}
Let $A\subset\C_*^n$ be the punctured cone on a connected submanifold $Y$ of $\P^{n-1}$, $n\geq 2$, and assume that $A$ is algebraically elliptic and not contained in a hyperplane in $\C^n$.  Let $M$ be a smooth affine curve and $\mathcal A$ be the subbundle of $(T^*M)^{\oplus n}$ defined by $A$.  Finally, let $K$ be a holomorphically convex compact subset of $M$ and $u:K\to\C^n$ be a holomorphic $A$-immersion.  Then the following are equivalent.
\begin{enumerate}[{\rm (i)}]
\item  $u$ can be uniformly approximated on a neighbourhood of $K$ by regular $A$-immersions $M\to\C^n$.
\item  $u$ can be uniformly approximated on a neighbourhood of $K$ by proper regular $A$-immersions $M\to\C^n$ agreeing with $u$ on any given finite set in $K$.
\item  There is a neighbourhood $U$ of $K$ such that the homotopy class of continuous sections of $\mathcal A\,\vert\, U$ that contains $du$ also contains the restriction of a regular section of $\mathcal A$ on $M$.
\item  There is a neighbourhood $U$ of $K$ such that the homotopy class of continuous sections of $\mathcal A\,\vert\, U$ that contains $du$ also contains the restriction of an exact regular section of $\mathcal A$ on $M$ with an effective pole at each end of $M$.
\end{enumerate}
\end{theorem}

\begin{remark}\label{rem:intro}
(a)  The only punctured cone in $\C^*$ is $\C^*$ itself, which is not algebraically elliptic, so Theorem \ref{t:main-theorem} is vacuous for $n=1$.  For $n=2$, the only punctured cone in $\C_*^2$ satisfying the assumptions of Theorem \ref{t:main-theorem} is $\C_*^2$ itself.  Our proof of Theorem \ref{t:main-theorem} uses several results from \cite{AlarconForstneric2014IM,AlarconCastro-Infantes2019APDE}, which are stated for holomorphic immersions directed by punctured cones $A\subset\C_*^n$ with $n\ge 3$ and $A\neq\C_*^n$.  Although an inspection of the proofs in \cite{AlarconForstneric2014IM, AlarconCastro-Infantes2019APDE} shows that the results that we apply are valid for $n\ge 2$ and $A=\C_*^n$ (in fact, extensions of some of these results are proved for arbitrary $n\ge 1$ in \cite{AlarconLarusson2017IJM} and \cite[Section 5]{ForstnericLarusson2019CAG}), we shall give a separate proof, much simpler, of this special case of Theorem \ref{t:main-theorem}.

(b)  Algebraic ellipticity may be formulated in several nontrivially equivalent ways.  It was observed in \cite{LarussonTruong2019} that the following properties of a smooth algebraic variety are equivalent: algebraic subellipticity, the algebraic homotopy Runge property, and the algebraic version of Gromov's Ell${}_1$ property (for the definitions of these properties, see \cite{LarussonTruong2019}).  And recently, Kaliman and Zaidenberg showed that algebraic ellipticity and algebraic subellipticity are equivalent \cite{KalimanZaidenberg2024}.

The optimal known sufficient condition for $A$ to be algebraically elliptic is given by a very recent theorem of Arzhantsev, Kaliman, and Zaidenberg \cite[Theorem 1.3]{ArzhantsevKalimanZaidenberg2024}.  They showed that $A$ is algebraically elliptic if $Y$ is uniformly rational, meaning that $Y$ is covered by Zariski-open sets, each isomorphic to a Zariski-open subset of affine space (see also \cite{BodnarEtAl2008BLMS}, where 
uniformly rational manifolds are called \lq\lq plain\rq\rq, and \cite{BogomolovBohning2014}).   
Then $Y$ itself is also algebraically elliptic \cite[Theorem 1.3]{ArzhantsevKalimanZaidenberg2024}.
The class of uniformly rational manifolds is closed under products and under blow-ups with smooth centres.  Also, the total space of an algebraic fibre bundle over a uniformly rational manifold with a uniformly rational fibre is itself uniformly rational.

Most (but not all) known examples of 
uniformly rational manifolds are in fact covered by Zariski-open sets isomorphic to affine space itself.  Such manifolds are said to be A-covered or of class $\mathscr A_0$ (see \cite[Section 4]{ArzhantsevPerepechkoSuss2014JLMS}).  The following are examples of manifolds of class $\mathscr A_0$:  projective spaces, Grassmannians, flag manifolds, smooth projective rational surfaces, and smooth projective toric varieties.

(c)  Approximating holomorphic maps by regular maps is an important theme in algebraic geometry.  Such approximation is impossible in general, even for maps between very simple affine algebraic manifolds.  For example, let $\Sigma^n$, 
$n\geq 2$, be the complex $n$-sphere $\{(z_0,\ldots,z_n)\in\C^{n+1}:z_0^2+\cdots+z_n^2=1\}$.  It is a homogeneous space of the connected linear algebraic group $\textrm{SO}(n+1, \C)$, which has no nontrivial characters, so $\Sigma^n$ is algebraically flexible \cite[Proposition 5.4]{ArzhantsevFlennerKalimanKutzschebauchZaidenberg2013DMJ} and hence algebraically elliptic.  Loday showed that every regular map $\Sigma^p\times \Sigma^q\to \Sigma^{p+q}$ is nullhomotopic when $p$ and $q$ are odd, but there is a continuous map, and hence a holomorphic map, $\Sigma^p\times \Sigma^q\to \Sigma^{p+q}$ that is not nullhomotopic (\cite{Loday1973}; see also \cite[Example 6.15.7]{Forstneric2017E}) and thus not approximable by regular maps.

(d)  The main task in the proof of Theorem \ref{t:main-theorem} is to show that (iii) implies (ii).  A key ingredient in our proof is Forstneri\v c's approximation theorem for sections of an algebraic fibre bundle with algebraically elliptic fibre \cite[Theorem 3.1]{Forstneric2006AJM} (see also \cite[Theorem 6.15.3]{Forstneric2017E}).  Applying this theorem requires a homotopy-theoretic necessary and sufficient condition, which in our setting corresponds to (iii).
\end{remark}

The following definitions are convenient.

\begin{definition}\label{def:good}
With notation as above, we say that a pair $(A, M)$ is {\em good\,} if $\mathcal A$ has a regular section on $M$, and we say that $(A, M)$ is {\em very good\,} if every homotopy class of continuous (or, equivalently, holomorphic) sections of $\mathcal A$ on $M$ contains a regular section.  Furthermore, we say that $A$ is {\em good\,} if $(A, M)$ is good for every smooth affine curve $M$, and we say that $A$ is {\em very good\,} if $(A, M)$ is very good for every smooth affine curve $M$. 
\end{definition}

\begin{remark}  \label{rem:serre}
An algebraically elliptic projective manifold is unirational and hence simply connected by Serre's \cite[Proposition 1]{Serre1959JLMS}.  Also, by the first half of \cite[Lemma~2]{Serre1959JLMS}, the fundamental group of an algebraically elliptic manifold is finite.  Thus, under the assumptions of Theorem \ref{t:main-theorem}, $\pi_1(Y)=0$ and $\pi_1(A)$ is finite.  Also, $\pi_1(A)$ is cyclic, being a quotient of the fundamental group of the fibre $\C^*$ of the bundle $A\to Y$. 

The tangent bundle $TM$ is holomorphically trivial, so continuous sections of $\mathcal A$ may be identified with continuous maps $M\to A$.  Since $M$ has the homotopy type of a finite bouquet of circles and deformation-retracts onto an embedded bouquet, the homotopy classes of continuous maps $M\to A$ are in bijective correspondence with homomorphisms $\pi_1(M) \to \pi_1(A)$.  Hence there are only finitely many such classes.  If $A$ is simply connected, then there is only one class, so $(A, M)$ is good if and only if it is very good.  In general, we are not aware of any pair $(A,M)$ as in Theorem \ref{t:main-theorem} that is good but not very good.
\end{remark}

In Section \ref{sec:cones}, we establish the following sufficient conditions for a cone to be good and a pair to be good or very good. 
%
%
\begin{proposition}  \label{pro:sufficient}
Let $Y$, $A$, and $M$ be as in Theorem \ref{t:main-theorem}.
\begin{enumerate}[{\rm (a)}]
\item  If $Y$ has a rational curve of degree 1 or 2, then $A$ is good.
\item  If $M$ has genus 0, then $(A, M)$ is very good.
\item  If $TM$ is algebraically trivial, for example if $M$ has genus 1 or is planar, then $(A,M)$ is good.
\end{enumerate}
\end{proposition}

Projective varieties of low codimension and defined by equations of low degree are covered by lines.  Among the instances of this metatheorem are \cite[Lemma V.4.8.1, Exercise V.4.10.5]{KollarRatCurvesOnAlgVar} and \cite[Theorem 3.1]{MarchesiMassarenti2014GeomDed}.  Similar results exist for conics in place of lines (see \cite[Proposition 2.4]{IonescuRusso2010Crelle}).  It is reasonable to ask whether every punctured cone $A$ as in Theorem \ref{t:main-theorem} is good.

Thus Proposition \ref{pro:sufficient}(a) applies to many cones of interest in the present context, such as those in the following examples.
%
%
\begin{example}\label{ex:}
(a) A first important example is the \lq\lq large cone\rq\rq\ $A=\C_*^n$, $n\geq 2$; see Remark \ref{rem:intro}(a). 
Every immersion into $\C^n$ is directed by $A$.  Theorem \ref{t:main-theorem} is new even for these simplest of all cones. 

(b) A cone of fundamental importance in the theory of minimal surfaces in Euclidean space and satisfying the hypotheses of Theorem \ref{t:main-theorem} is the punctured null quadric 
$A=\{(z_1, \ldots, z_n)\in\C_*^n: z_1^2+\cdots+ z_n^2=0\}$, $n\ge 3$.  For this cone, Theorem \ref{t:main-theorem} was proved in \cite{AlarconLopez2022APDE} by a rather involved method that heavily uses the special geometry of the cone.  The initial motivation for the present work was to generalise this result to the largest possible class of cones, while giving an easier, more conceptual proof than the proof in \cite{AlarconLopez2022APDE}.

(c) The cone $A=\{(z_1,\ldots,z_4)\in\C_*^4: z_1z_4-z_2z_3=0\}$ is isomorphic to the punctured null quadric in $\C_*^4$ and satisfies the requirements of Theorem \ref{t:main-theorem}.  It plays a key role in the theory of surfaces of constant mean curvature $1$ in hyperbolic space.  Bryant \cite{Bryant1987Asterisque} calls a holomorphic immersion of a Riemann surface into the special linear group $\textrm{SL}_2(\C)=\{(z_1,\ldots,z_4)\in\C^4: z_1z_4-z_2z_3=1\}$, directed by $A$, a null curve.  The projection of a holomorphic null curve in $\textrm{SL}_2(\C)$ to the hyperbolic 3-space $\mathbb H^3=\textrm{SL}_2(\C)/\textrm{SU}(2)$ is a conformally immersed surface of constant mean curvature $1$ (also called a Bryant surface), and every simply connected surface in $\mathbb H^3$ of constant mean curvature $1$ arises in this way.  We refer to \cite{UmeharaYamada1993AM,CollinHauswirthRosenberg2001AM} and the introduction of the more recent paper \cite{AlarconCastroHidalgo2025} for background on this subject.  It is an open question whether every open Riemann surface admits a proper holomorphic null immersion or embedding into $\textrm{SL}_2(\C)$ \cite[Problem 1]{AlarconForstneric2015MA}. The same question in the algebraic category is open as well.
\end{example}

In Section \ref{sec:cones}, we prove that all the cones in Example \ref{ex:} are in fact very good and discuss their interest in the literature. Theorem \ref{t:main-theorem} and the following corollaries, concerning good and very good pairs, are proved in Section \ref{sec:proofs}; see Remark \ref{rem:vj} for a slight extension of these results. 
%
%
\begin{corollary}   \label{c:second-corollary}
Under the assumptions of Theorem \ref{t:main-theorem}, the following are equivalent.
\begin{enumerate}[{\rm (i)}]
\item  There is a regular $A$-immersion $M\to\C^n$.
\item  For any finite set $\Lambda\subset M$, every map $\Lambda\to\C^n$ extends to a proper regular $A$-immersion $M\to\C^n$.
\item  $(A,M)$ is good.
\item  There is an exact regular section of $\mathcal A$ on $M$ with an effective pole at each end of $M$.
\end{enumerate}
\end{corollary}
%
%
\begin{corollary}   \label{c:main-corollary}
Under the assumptions of Theorem \ref{t:main-theorem}, the following are equivalent.
\begin{enumerate}[{\rm (i)}]
\item  Every holomorphic $A$-immersion $K\to\C^n$, where $K\subset M$ is compact and holomorphically convex, can be uniformly approximated on $K$ by regular $A$-immersions $M\to\C^n$.
\item  Every holomorphic $A$-immersion $K\to\C^n$, where $K\subset M$ is compact and holomorphically convex, can be uniformly approximated on $K$ by proper regular $A$-immersions $M\to\C^n$ agreeing with the given holomorphic $A$-immersion $K\to\C^n$ on any given finite set in $K$.
\item  $(A,M)$ is very good.
\item  Every homotopy class of continuous sections of $\mathcal A$ on $M$ contains an exact regular section with an effective pole at each end of $M$.
\end{enumerate}
\end{corollary}
For any open Riemann surface $M$ and any punctured cone $A\subset\C^n$, $n\geq 3$, that is Oka and not contained in a hyperplane, Alarc\'on and Forstneri\v c proved a general existence result for proper holomorphic $A$-embeddings $M\to\C^n$ in \cite{AlarconForstneric2014IM}.  For $n=2$, it is an open question whether every open Riemann surface admits a proper holomorphic embedding into $\C^2$ \cite{Forster1970,BellNarasimhan1990}. The problem seems very difficult even for smooth affine curves; see \cite[Sections 9.10--9.11]{Forstneric2017E}.  In the algebraic setting, it is well known that every smooth affine curve regularly embeds in $\C^3$, while most of them do not regularly embed in $\C^2$.  Except for the large cones, even for the punctured null quadrics in Example \ref{ex:}(b), the following question remains open.

\begin{question}
Let $A$ and $M$ be as in Theorem \ref{t:main-theorem} with $A\neq \C_*^n$ and assume that the pair $(A,M)$ or the cone $A$ is good or even very good.  Is there a proper regular $A$-embedding $M\to\C^n$?
\end{question}
A main difficulty in extending the results in this paper to embeddings is that one cannot use the standard transversality method that has proved useful in the holomorphic category (see \cite[Section 6]{AlarconForstneric2014IM}).  Indeed, this method would allow us to eliminate the self-intersections in a compact piece of the affine curve, but to make it embedded one would need to apply it in a recursive way, which does not seem compatible with our approach.

%
%
Our method of proof also gives results analogous to Theorem \ref{t:main-theorem} for directed harmonic maps into $\R^n$, $n\ge 3$, including conformal minimal immersions.  We explain this in Section \ref{sec:harmonic}.  In fact these results follow from a more general approximation theorem for regular sections of $\mathcal A$ with control on their periods, which we state and prove in Section \ref{sec:sections} as Theorem \ref{t:sections}.  The following result, saying in particular that the periods of a regular section of $\mathcal A$ on $M$ can be prescribed arbitrarily whenever the pair $(A,M)$ is good, is an example of the consequences of Theorem \ref{t:sections}.  As far as we know, the result is new even for the \lq\lq large cones\rq\rq\ in Example \ref{ex:}(a).

%
%
\begin{theorem}\label{t:cohomology}
Let $A$, $M$, and $\mathcal A$ be as in Theorem \ref{t:main-theorem} and assume that $t_0$ is a regular section of $\mathcal A$ on $M$. Then for any group homomorphism $\mathcal F:H_1(M,\Z)\to\C^n$ there is a regular section $t$ of $\mathcal A$ on $M$ homotopic to $t_0$ such that $t$ has an effective pole at each end of $M$ and $\int_C t=\mathcal F(C)$ for every loop $C$ in $M$.

In particular, the following hold for any pair $(A,M)$ as in Theorem \ref{t:main-theorem}.
\begin{enumerate}[A]
\item[$\bullet$] $(A,M)$ is good if and only if every class in the cohomology group $H^1(M,\C^n)$ contains a regular section of $\mathcal A$ on $M$ with an effective pole at each end of $M$. 
\item[$\bullet$] $(A,M)$ is very good if and only if every class in the cohomology group $H^1(M,\C^n)$ contains a regular section of $\mathcal A$ on $M$ with an effective pole at each end of $M$ in each homotopy class of continuous sections of $\mathcal A$ on $M$.
\end{enumerate}
\end{theorem}

See Theorem \ref{t:cohomology-2} for a slight extension of this result. Theorem \ref{t:cohomology} strengthens the implications (iii) $\Rightarrow$ (iv) in Theorem \ref{t:main-theorem} and Corollaries \ref{c:second-corollary} and \ref{c:main-corollary}.  There is a long history of results about cohomology classes containing representatives with special properties.  The holomorphic analogue of Theorem \ref{t:cohomology} is \cite[Corollary 2(a)]{AlarconLarusson2017IJM}.  It does not require the assumption that $(A, M)$ is good or any similar hypothesis.
%
%
%


\section{Good and very good pairs and cones}
\label{sec:cones}

\noindent
In this section we first prove Proposition \ref{pro:sufficient}, providing sufficient conditions for a cone to be good or for a pair to be very good in the sense of Definition \ref{def:good}. Then we show that the cones presented in Example \ref{ex:} are very good.
%
%
\begin{proof}[Proof of Proposition \ref{pro:sufficient}]
(a)  Let $f:\P^1\to \P^{n-1}$, $[z,w]\mapsto [q_1(z,w), \ldots, q_n(z,w)]$, be a rational curve of degree 1 in $Y$.  Here, $q_1,\ldots,q_n$ are homogeneous polynomials of degree 1.  Let $\omega$ and $\eta$ be regular forms on $M$ with no common zeros (such a pair exists by Riemann-Roch).  Then $(q_1(\omega,\eta),\ldots,q_n(\omega,\eta))$ is a regular section of~$\mathcal A$.  If $f$ has degree 2, then we apply the same argument to a pair of regular sections on $M$ with no common zeros of a square root of $T^*\overline M$.

(b)  Since $TM$ is algebraically trivial, regular sections of $\mathcal A$ may be identified with regular maps $M\to A$.  We need to show that every homomorphism $\pi_1(M)\to\pi_1(A)$ is induced by a regular map (see Remark \ref{rem:serre}).  Consider the piece $\pi_1(\C^*) \to \pi_1(A) \to \pi_1(Y)=0$ of the long exact sequence of homotopy groups associated to the projection $A\to Y$.  By the universal property of the free group $\pi_1(M)$, every homomorphism $\pi_1(M)\to \pi_1(A)$ factors through the homomorphism $\pi _1(\C^*) \to \pi_1(A)$.  Since $M$ has genus 0, every homomorphism $\pi_1(M)\to \pi_1(\C^*)$ is induced by a regular map $M\to \C^*$ and the proof is complete.

(c)  $\mathcal A$ has a regular section corresponding to a constant map $M\to A$.
\end{proof}

\subsection{The large cone}\label{ss:large}
We begin by proving that the large cone $A=\C_*^n$, $n\ge 2$, is very good.  In this case, an $A$-immersion is nothing but an immersion.  We give a self-contained proof, not relying on Theorem \ref{t:main-theorem} (see Remark \ref{rem:intro}(a)).  We first show that condition (ii) in Theorem \ref{t:main-theorem} always holds for the large cone, even including jet interpolation.
%
%
\begin{proposition}\label{pro:large}
Let $M$ be a smooth affine curve and $K$ be a holomorphically convex compact subset $K$ of $M$.  Then every holomorphic immersion $K\to\C^n$, $n\ge 2$, may be approximated uniformly on a neighbourhood of $K$ by proper regular immersions $M\to\C^n$ agreeing with the given immersion to any given finite order at any given finite subset of $K$.  Furthermore, we can choose the approximating regular immersions such that all their component functions are proper.
\end{proposition}
\begin{proof}
 For simplicity of exposition we shall assume that $n=2$.  The proof for $n\ge 3$ follows the same lines.  So assume that $K\subset M=\overline M\setminus\{p_1,\ldots,p_m\}$ is a holomorphically convex compact subset and $u=(u_1,u_2):K\to \C^2$ is a holomorphic immersion. Also, for the jet interpolation condition, let $\Lambda\subset K$ be a finite set and $d\ge 1$ be an integer. Choose a smoothly bounded connected compact domain\footnote{It is customary to call a nonempty set in a topological space a compact domain if it is compact and is the closure of a connected open set.  Lacking a better term, in this paper we also refer to the union of finitely many mutually disjoint compact domains as a compact domain.} $L\supset K$ in $M$ such that $\overline M\setminus \overset\circ L = D_1\cup\cdots\cup D_m$, where $D_1,\ldots, D_m$ are mutually disjoint smoothly bounded closed discs centred at $p_1,\ldots, p_m$, respectively.  By the classical Runge theorem with jet interpolation we may approximate $u$ uniformly on a neighbourhood of $K$ by a holomorphic map $L\to\C^2$ agreeing with $u$ to order $d$ at all points of $\Lambda$.  By general position, this map can be chosen to be an immersion on a neighbourhood of $L$.  Thus we may assume without loss of generality that $u$ extends to $L$ as a holomorphic immersion $u=(u_1,u_2):L\to \C^2$. 
 
Fix a finite set $F=\{q_1,\ldots,q_m\}$ with $q_i\in \mathring D_i\setminus\{p_i\}$ for $i=1,\ldots,m$.  By Royden's Runge theorem for regular functions \cite[Theorem 10]{Royden1967JAM}, we may approximate $u_1$ uniformly on a neighbourhood of $L$ by a regular function $v_1:M\to\C$ agreeing with $u_1$ to order $d$ at all points of $\Lambda$ and satisfying
\begin{equation}\label{eq:v1-proper}
	\min\{|v_1(p)|: p\in F\}>\max\{|v_1(p)|: p\in bL\}.
\end{equation}
Assuming that the approximation is close enough, the holomorphic map $(v_1,u_2):L\to\C^2$ is still an immersion on a neighbourhood of $L$.  Let $B$ denote the set of critical points of $v_1$ on $M\setminus L$.  Since $v_1$ is regular, $B$ is finite.  Also choose for $i=1,\ldots,m$ a smoothly bounded closed disc $\Delta_i\subset \mathring D_i\setminus\{p_i\}$ such that $B\subset\Delta=\Delta_1\cup\cdots\cup \Delta_m$.  Extend $u_2:L\to\C$ to a holomorphic map $u_2:L\cup\Delta\to\C$ such that $u_2$ is an immersion on $\Delta$ and
\begin{equation}\label{eq:MP2}
	\min\{|u_2(p)|: p\in \Delta\}>\max\{|u_2(p)|: p\in bL\}.
\end{equation}
(Recall that $L\cap \Delta=\varnothing$ and $L$ and $\Delta$ are compact.)  Again by \cite[Theorem 10]{Royden1967JAM}, we may approximate $u_2$ uniformly on a neighbourhood of $L\cup\Delta$ by a regular function $v_2$ on $M$ agreeing with $u_2$ to order $d$ at all points of $\Lambda$.  If the approximation is close enough, then $v=(v_1,v_2):M\to\C^2$ is still an immersion on a neighbourhood of $L$, while $v_2:M\to\C$ is an immersion on a neighbourhood of $\Delta$.  Since $v_1:M\to\C$ is an immersion on $M\setminus (L\cup B)$ and $B\subset \Delta$, we infer that $v:M\to\C^2$ is a regular immersion.  Note that $v$ is close to $u$ uniformly on a neighbourhood of $K$ and agrees with $u$ to order $d$ at all points of $\Lambda$.  Finally, assuming that $v_2$ is sufficiently close to $u_2$ uniformly on $L\cup \Delta$, condition \eqref{eq:MP2} ensures that 
\[ \min\Big\{|v_2(p)|: p\in \Delta=\bigcup_{i=1}^m\Delta_i\Big\}>\max\Big\{|v_2(p)|: p\in bL=\bigcup_{i=1}^m bD_i\Big\}.  
\]
We claim that $v_2$ has a pole at $p_i$ for all $i=1,\ldots,m$, so $v_2:M\to\C$ is a proper map.  Indeed, otherwise, $v_2$ would be holomorphic on the smoothly bounded closed disc $D_i$, violating the maximum modulus principle since $\Delta_i\subset\mathring D_i$ and $\min\{|v_2(p)|: p\in \Delta_i\}>\max\{|v_2(p)|: p\in bD_i\}$.  Using \eqref{eq:v1-proper}, the same argument shows that $v_1:M\to\C$ is a proper map. In particular, $v:M\to \C^2$ is proper.
\end{proof}

%
%
\begin{corollary}\label{co:large-very-good}
If $A=\C_*^n$, $n\ge 2$, $M$ is a smooth affine curve, and $\mathcal A$ is the subbundle of $(T^*M)^{\oplus n}$ defined by $A$, then conditions {\rm (i)}--{\rm (iv)} in Corollary \ref{c:main-corollary} hold. In particular, $A$ is very good in the sense of Definition \ref{def:good}.
\end{corollary}
\begin{proof}
Proposition \ref{pro:large} implies (ii), while it is obvious that (ii) implies (i) and that (iv) implies (iii). So it remains to check (iv).  For this, in view of Proposition \ref{pro:large}, it suffices to show that every homotopy class of continuous or, equivalently, holomorphic sections of $\mathcal A$ over $M$ contains $du$ for some holomorphic immersion $u:M\to\C^n$.  By \cite[Theorems 5.3 and 5.4]{ForstnericLarusson2019CAG} (see also \cite[Theorem 1]{AlarconLarusson2017IJM}), every holomorphic section of $\mathcal A$ over $M$ can be deformed to an exact holomorphic section, that is, a section of the form $du$.
\end{proof}

%
%
\subsection{The null quadric}\label{ss:null}
This is the cone
\begin{equation}\label{eq:null}
	A=\{(z_1, \ldots, z_n)\in\C_*^n: z_1^2+\cdots+ z_n^2=0\},\quad n\ge 3.
\end{equation}
This cone is algebraically elliptic as seen in \cite[Proposition 1.15.3]{AlarconForstnericLopez2021}.  A holomorphic immersion $M\to\C^n$ directed by $A$ is called a null curve and its real part $M\to\R^n$ is a conformal minimal immersion.  Conversely, every conformal minimal immersion $M\to\R^n$ is the real part of a holomorphic null curve on any simply connected domain in $M$.  The natural counterpart of a regular function in the theory of minimal surfaces is a complete minimal surface of finite total curvature.  Indeed, if $M$ is an open Riemann surface and $X:M\to\R^n$ is a complete conformal minimal immersion of finite total curvature, then $M$ is a smooth affine curve and the $(1,0)$-part of the exterior derivative of $X$ is a regular section of $\mathcal A$ on $M$ with an effective pole at each end of $M$.  (See the classical surveys \cite{Osserman1986,BarbosaColares1986LNM,Yang1994MA} or the recent monograph \cite{AlarconForstnericLopez2021}, in particular Chapter 4.)  For regular null curves and complete minimal surfaces of finite total curvature, Theorem \ref{t:main-theorem} was recently proved by Alarc\'on and L\'opez in \cite{AlarconLopez2022APDE} (see \cite{Lopez2014TAMS, AlarconCastro-InfantesLopez2019CVPDE} for the case $n=3$). 
By the results in \cite{AlarconLopez2022APDE}, the punctured null quadric $A$ \eqref{eq:null} is very good.  The proofs in \cite{Lopez2014TAMS,AlarconCastro-InfantesLopez2019CVPDE,AlarconLopez2022APDE}, relying mainly on the theory of Riemann surfaces, are much more involved than the arguments here and heavily use the special geometry of the null quadric.  It does not seem that these proofs could be adapted to more general families of regular directed immersions.  For completeness, we give a simple proof that the punctured null quadric is very good. 
(This can also be seen as an immediate corollary of the results in \cite{Lopez2014TAMS,AlarconLopez2022APDE}.)  Thus, if $M$ is a smooth affine curve and $\mathcal A$ is the subbundle of $(T^*M)^{\oplus n}$ defined by $A$, then conditions (i)--(iv) in Corollary \ref{c:main-corollary} hold.
%
%
\begin{proposition}\label{pro:null-very-good}
The punctured cone $A=\{(z_1, \ldots, z_n)\in\C_*^n: z_1^2+\cdots+ z_n^2=0\}$, $n\ge 3$,  is very good in the sense of Definition \ref{def:good}.
\end{proposition}
\begin{proof}
By Proposition \ref{pro:sufficient}(a), $A$ is good. If $n\geq 4$, then $A$ is simply connected and hence very good by Remark \ref{rem:serre}.  In case $n=3$, $A$ is 
isomorphic to 
\[
	A'=\{z=(z_1,z_2,z_3)\in\C_*^3 : z_1z_2=z_3^2\}
\] 
and is not simply connected (its fundamental group is $\Z_2$: see \cite[Equation (8.3)]{AlarconForstneric2018Crelle}).  We will prove that the pair $(A',M)$ is very good for every smooth affine curve $M$.  Denote by $\mathcal A'$ the subbundle of $(T^*M)^{\oplus 3}$ defined by $A'$. Since $A'$ is an Oka manifold, every homotopy class of continuous sections of $\mathcal A'$ on $M$ contains a holomorphic section. Thus it suffices to see that for any holomorphically convex compact set $K\subset M$, which is a strong deformation retract of $M$, and for any holomorphic section $\eta$ of $\mathcal A'$ on a neighbourhood of $K$, we can approximate $\eta$ on a neighbourhood of $K$ by a regular section of $\mathcal A'$ on $M$. 
%
%
For this, note that $\eta=(\eta_1,\eta_2,\eta_3)$, where each $\eta_i$ is a holomorphic 1-form on a neighbourhood of $K$ such that 
\begin{equation}\label{eq:eta1eta2}
	\eta_1\eta_2=\eta_3^2 
\end{equation}
and $\eta_1$ and $\eta_2$ have no common zeros on $K$. 
By \eqref{eq:eta1eta2}, $\eta_1$ and $\eta_2$ are spinorially equivalent on a neighbourhood of $K$, meaning that there is a meromorphic function $f$ on a neighbourhood of $K$ such that $\eta_1=f^2\eta_2$. Choose a regular 1-form $\omega$ on $M$ satisfying the following conditions (see \cite[Lemma 3.2]{AlarconCastro-InfantesLopez2019CVPDE}).
\begin{enumerate}[(a)]
\item $\omega$ has no zeros on $K$.
\item $\omega$ is spinorially equivalent to $\eta_1$, hence also to $\eta_2$, on a neighbourhood of $K$.
\end{enumerate}
By (b), there is a meromorphic function $f_i$ on a neighbourhood of $K$ such that 
\begin{equation}\label{eq:etai}
	\eta_i=f_i^2\omega,\quad i=1,2.
\end{equation} 
Note that (a) ensures that $f_i$ is in fact holomorphic on a neighbourhood of $K$, $i=1,2$.  By \eqref{eq:eta1eta2} and after replacing $f_1$ by $-f_1$ if necessary, we may assume that
\begin{equation}\label{eq:f1f2}
	\eta_3=f_1f_2\omega\quad \text{on a neighbourhood of }K. 
\end{equation} 
 Denote by $Z_1$ the zero set of $\omega$, a finite subset of $M\setminus K$ (see (a)). Let $\Omega_1\subset M\setminus K$ be a compact neighbourhood of $Z_1$ such that each component of $\Omega_1$ is a closed disc containing a single point in $Z_1$. 
Choose a meromorphic function $f_1$ on a neighbourhood of $\Omega_1$ such that $f_1^2\omega$ is holomorphic and has no zeros on a neighbourhood of  $\Omega_1$. Fix a number $\epsilon_1>0$. By Royden's Runge theorem for regular functions \cite[Theorem 10]{Royden1967JAM}, there is a regular function $g_1$ on $M\setminus Z_1$ satisfying the following conditions.
\begin{enumerate}[(A{1})]
\item $g_1/f_1$ is holomorphic and has no zeros on $K$.
\item $|g_1-f_1|<\epsilon_1$ on a neighbourhood of $K\cup\Omega_1$.
\item $g_1^2\omega$ is regular on $M$ and has no zeros on $\Omega_1$.
\end{enumerate}

Now, let $Z_2$ denote the zero set of $g_1$ on $M\setminus K$. Note that $Z_2$ is finite and $Z_2\subset M\setminus(K\cup\Omega_1)$ by (C1). Let $\Omega_2\subset M\setminus(K\cup\Omega_1)$ be a compact neighbourhood of $Z_2$ such that each component of $\Omega_2$ is a closed disc containing a single point in $Z_2$. Extend the function $f_2$ to $\Omega_1\cup\Omega_2$ by setting $f_2=g_1$ on $\Omega_1$ and $f_2=1$ on $\Omega_2$; note that the compact sets $K$, $\Omega_1$, and $\Omega_2$ are mutually disjoint.  Also note that $f_2$ is meromorphic on a neighbourhood of $K\cup\Omega_1\cup\Omega_2$ and holomorphic on $(K\cup\Omega_1\cup\Omega_2)\setminus Z_1$. Choose $\epsilon_2>0$. By \cite[Theorem 10]{Royden1967JAM}, there is a regular function $g_2$ on $M\setminus Z_1$ satisfying the following conditions.
\begin{enumerate}[(A{2})]
\item $g_2/f_2$ is holomorphic and has no zeros on $K$.
\item $|g_2-f_2|<\epsilon_2$ on a neighbourhood of $K\cup\Omega_1\cup\Omega_2$.
\item $g_2^2\omega$ is regular on $M$ and has no zeros on $\Omega_1\cup\Omega_2$. (This is possible by (C1).)
\end{enumerate}

Set $\rho_i=g_i^2\omega$, $i=1,2$. By (C1), (C2), and since $Z_1\subset \Omega_1$, these are regular forms on $M$. By (a), (A1), (A2), and since $\eta_1$ and $\eta_2$ have no common zeros on $K$, we see that $\rho_1$ and $\rho_2$ have no common zeros on $K$. Using also (C1), (C2), and since $Z_1\subset\Omega_1$ and $Z_2\subset\Omega_2$, we then infer that $\rho_1$ and $\rho_2$ have no common zeros on $M$. Finally, consider the regular form $\rho_3=g_1g_2\omega$ on $M$. It is clear that $\rho_1\rho_2=\rho_3^2$, so $\rho=(\rho_1,\rho_2,\rho_3)$ is a regular section of $\mathcal A'$ on $M$. Finally, \eqref{eq:etai}, \eqref{eq:f1f2}, (B1), and (B2) ensure that $\rho$ is $\epsilon$-close to $\eta$ on a neighbourhood of $K$ for any given $\epsilon>0$, provided that $\epsilon_1$ and $\epsilon_2$ are chosen sufficiently small. 
\end{proof}


\section{Approximation by regular sections}
\label{sec:sections}

\noindent
In this section we establish the following theorem, the second main result of the paper, on approximation by regular sections with control of periods.  Subsequently, we prove Theorem \ref{t:cohomology}. Recall that the first homology group with integer coefficients $H_1(M,\Z)$ of a smooth affine curve $M$ is isomorphic to $\Z^r$ for some integer $r\ge 0$.
%
%
\begin{theorem}   \label{t:sections}
Let $A\subset\C_*^n$ be the punctured cone on a connected submanifold $Y$ of $\P^{n-1}$, $n\geq 2$, and assume that $A$ is algebraically elliptic and not contained in a hyperplane in $\C^n$.  Let $M$ be a smooth affine curve and $\mathcal A$ be the subbundle of $(T^*M)^{\oplus n}$ defined by $A$.  Also let $K$ be a holomorphically convex compact subset of $M$, $s$ be a holomorphic section of $\mathcal A$ on a neighbourhood of $K$, and $\mathcal F:H_1(M,\Z)\to\C^n$ be a group homomorphism such that 
\begin{equation}\label{eq:periods-s}
	\mathcal F(C)=\int_C s\quad \text{for every loop $C$ in a neighbourhood of $K$.}
\end{equation}
Finally, let $C_1,\ldots,C_\ell$ be smooth oriented embedded arcs in $K$ such that $C_1\cup\cdots\cup C_\ell$ is holomorphically convex in $M$ and each arc $C_l$ contains a nontrivial arc disjoint from $\bigcup_{i\neq l} C_i$. Then the following are equivalent.
\begin{enumerate}[{\rm (i)}]
\item  $s$ can be uniformly approximated on a neighbourhood of $K$ by regular sections $t$ of $\mathcal A$ on $M$ with an effective pole at each end of $M$,
such that $\int_C t=\mathcal F(C)$ for every loop $C$ in $M$ and $\int_{C_l} t=\int_{C_l} s$ for $l=1,\ldots,\ell$.
\item  There is a neighbourhood $U$ of $K$ such that the homotopy class of continuous sections of $\mathcal A\,\vert\, U$ that contains $s$ also contains the restriction of a regular section of $\mathcal A$ on $M$.
\end{enumerate}
Furthemore, if {\rm (ii)} holds, then the regular sections $t=(t_1,\ldots,t_n)$ in {\rm (i)} can be chosen such that all the component 1-forms $t_1,\ldots,t_n$ have an effective pole at each end of~$M$.
\end{theorem}

We start with some preparations. We expect the following to have been observed before, but we do not know a reference for it.
%
%
\begin{lemma}\label{lem:extension}
Let $M$ be an open Riemann surface and $Y$ be a connected complex manifold.  Let $K$ be a holomorphically convex compact set in $M$ and $f:K\to Y$ be a continuous map.  Then $f$ admits a continuous extension $M\to Y$.  If, in addition, $Y$ is an Oka manifold and $f$ is holomorphic in $\mathring K$, then $f$ can be approximated uniformly on $K$ by holomorphic maps $M\to Y$ homotopic to any given continuous extension $M\to Y$ of $f$.
\end{lemma}
\begin{proof}
By the Tietze extension theorem and the existence of a smooth tubular neighbourhood of $Y$ once smoothly embedded into some Euclidean space, $f$ extends continuously to a neighbourhood of $K$.  Let $K_0\supset K$ be a smoothly bounded holomorphically convex compact domain in $M$, possibly disconnected, such that $f$ admits a continuous extension $f_0:K_0\to Y$. Let $\rho$ be a smooth strongly subharmonic Morse exhaustion function on $M$ with $K_0=\{\rho\le 0\}$. Let $a_0=0<a_1<a_2<\cdots$ be a divergent sequence of regular values of $\rho$ such that $\rho$ has at most one critical point in $K_{i+1}\setminus K_i$ for all $i=0,1,2,\ldots$, where 
$K_i=\{\rho\le a_i\}$. Each $K_i$ is a smoothly bounded compact domain (possibly disconnected) with $K_i\Subset K_{i+1}$ and $M=\bigcup_{i\ge 0} K_i$. If $\rho$ does not have critical points in $K_1\setminus K_0$, then $K_0$ is a strong deformation retract of $K_1$, and hence $f_0$ (and thus $f$) extends continuously to $K_1$. If there is a (necessarily unique) critical point of $\rho$ in $K_1\setminus K_0$, then there is a smooth embedded Jordan arc $E$ in $M$ intersecting $K_0$ precisely at its two endpoints and such that the intersection of $E$ and the boundary of $K_0$ is transverse and $K_0\cup E$ is a strong deformation retract of $K_1$. Since $Y$ is path-connected, $f_0$ extends to a continuous map $K_0\cup E\to Y$, and since $K_0\cup E$ is a strong deformation retract of $K_1$, there is a continuous extension $f_1:K_1\to Y$ of $f_0$. Repeating this process we may then extend $f_1$ to a continuous map $f_2:K_2\to Y$, and continuing inductively we obtain in the limit a continuous extension $M\to Y$ of $f_0$, hence of $f$. This proves the first assertion in the lemma. The second assertion is then guaranteed by the Mergelyan theorem for maps from Riemann surfaces to Oka manifolds; see e.g.\ \cite[Corollary 8]{FornaessForstnericWold2020}.
\end{proof}

A set $S$ in a smooth surface $M$ is said to be {\em admissible} if it is of the form $S=K\cup E$, where $K$ is a (possibly empty) finite union of mutually disjoint compact connected domains with piecewise $\Cscr^1$ boundaries in $M$ and $E=S\setminus \mathring K$ is a finite union of mutually disjoint smooth Jordan arcs and closed Jordan curves meeting $K$ only at their endpoints (if at all) and such that their intersections with the boundary $bK$ of $K$ are transverse; see \cite[Definition 1.12.9]{AlarconForstnericLopez2021}. Let us record the following result.
%
%
\begin{lemma}\label{lem:3.3.1}
Let $A\subset\C_*^n$ be the punctured cone on a connected submanifold $Y$ of $\P^{n-1}$, $n\geq 2$, and assume that $A$ is an Oka manifold and is not contained in a hyperplane in $\C^n$.
Let $M$ be an open Riemann surface, $S=K\cup E\subset M$ be a holomorphically convex admissible set, and $C_1,\ldots,C_\ell$ be smooth oriented embedded Jordan arcs and closed Jordan curves  in $S$ such that $C_1\cup\cdots\cup C_\ell$ is holomorphically convex in $M$ and each $C_l$ contains a nontrivial arc disjoint from $\bigcup_{i\neq l} C_i$.  Finally, let $\theta$ be a holomorphic $1$-form vanishing nowhere on $M$. 
Then every map $f:S\to A$ which is of class $\Cscr^r$ $(r\ge 0)$ on a neighbourhood of $S$ and holomorphic on $\mathring S=\mathring K$ can be approximated in $\Cscr^r(S)$ by holomorphic maps $F:M\to A$ such that $F(M)$ is not contained in a hyperplane in $\C^n$, $F$ agrees with $f$ on any given finite set in $S$, $F$ agrees with $f$ to any given finite order on any given finite set in $\mathring S$, and $\int_{C_l} F\theta=\int_{C_l}f\theta$ for all $l=1,\ldots,\ell$.
\end{lemma}
Note that the given map $f$ in the lemma extends to a continuous map $M\to A$ by Lemma \ref{lem:extension}, and hence there is no topological obstruction to holomorphic approximation.  When $A$ is the punctured null quadric \eqref{eq:null}, Lemma \ref{lem:3.3.1} coincides with \cite[Lemma 3.3.1]{AlarconForstnericLopez2021}. The proof for an arbitrary cone $A$ as in the statement follows the proof in \cite{AlarconForstnericLopez2021} word for word, but using the tools from \cite{AlarconForstneric2014IM,AlarconCastro-Infantes2019APDE} that deal with maps into an arbitrary cone. See also \cite[Lemma 4.2]{AlarconCastro-Infantes2019APDE} for an analogous result in a slightly different topological setting. We leave the details of the proof to the reader.

%
%
\begin{lemma}\label{lem:paths}
Let $A\subset\C_*^n$ be the punctured cone on a connected submanifold $Y$ of $\P^{n-1}$, $n\geq 2$, and assume that $A$ is an Oka manifold, not contained in a hyperplane in $\C^n$. Also let $\sigma:[0,1]\to A$ be a path. Then for any $\mu\in\C^n$, there is a homotopy of paths $\sigma^t:[0,1]\to A$, $t\in[0,1]$, such that $\sigma^0=\sigma$, $\sigma^t(0)=\sigma(0)$ and $\sigma^t(1)=\sigma(1)$ for all $t\in[0,1]$, 
and $\int_0^1\sigma^1(x)\,dx=\mu$.
\end{lemma}
\begin{proof}
Applying a small deformation to $\sigma$ keeping $\sigma(0)$ and $\sigma(1)$ fixed, we can assume that $\sigma([0,1])$ is not contained in any hyperplane in $\C^n$. This enables us to use the analogue of \cite[Lemma 3.1]{ForstnericLarusson2019CAG} with the punctured null quadric replaced by the cone $A$ (see \cite[p.\ 10]{AlarconLarusson2017IJM} or \cite[proof of Theorem 5.3]{ForstnericLarusson2019CAG}) to obtain a homotopy $\sigma^t$ satisfying the conclusion of the lemma. Note that the approximate condition in \cite[Lemma 3.1]{ForstnericLarusson2019CAG} can be made exact since $\sigma([0,1])$ is not contained in any hyperplane in $\C^n$, as is seen in \cite[proof of Theorem 1]{AlarconLarusson2017IJM} and \cite[proof of Theorem 5.3]{ForstnericLarusson2019CAG}. We also refer to \cite[Lemma 3.3]{AlarconCastro-Infantes2019APDE} for a closely related result. 
\end{proof}

With Lemmas \ref{lem:3.3.1} and \ref{lem:paths} in hand, we are now ready to prove the following first main technical step in the proof of Theorem \ref{t:sections}.
%
%
\begin{proposition}\label{p:Oka}
Let $A$ and $M$ be as in Theorem \ref{t:sections}. Also let $\mathcal F:H_1(M,\Z)\to\C^n$ be a group homomorphism and $\theta$ be a nowhere-vanishing holomorphic $1$-form on $M$. 
Then every continuous map $f:M\to A$ is homotopic to a holomorphic map $\tilde f:M\to A$ such that $\tilde f(M)$ is not contained in a hyperplane in $\C^n$ and 
\[
	\int_C\tilde f\theta=\mathcal F(C)\quad \text{for every loop $C$ in $M$}.
\] 
Furthermore, if $f$ is holomorphic on a neighbourhood of a holomorphically convex compact set $K$ in $M$, 
\begin{equation}\label{eq:loop-F}
	\int_C f\theta=\mathcal F(C)\quad \text{for every loop $C$ in a neighbourhood of $K$,} 
\end{equation}
and $C_1,\ldots,C_\ell$ are smooth oriented embedded arcs in $K$ such that  $C_1\cup\cdots\cup C_\ell$ is holomorphically convex in $M$ and each arc $C_l$ contains a nontrivial arc disjoint from $\bigcup_{i\neq l} C_i$, then a holomorphic map $\tilde f$ as above can be chosen to approximate $f$ uniformly on a neighbourhood of $K$, to agree with $f$ to any given finite order on any given finite set in $K$, and to satisfy 
\[
	\int_{C_l}\tilde f\theta=\int_{C_l}f\theta \quad \text{for } l=1,\ldots,\ell.
\]
\end{proposition}
We point out that the interpolation condition in this proposition (or in Lemma \ref{lem:3.3.1}) will not be required for further developments in this paper. In particular, it is not used to prove Theorem \ref{t:main-theorem}.
%
%
\begin{proof}
We may assume without loss of generality that $K\neq\varnothing$ is a smoothly bounded compact domain, possibly disconnected. 

We claim that $K$ can be assumed to be connected. Indeed, otherwise denote by $K_0,K_1,\ldots,K_N$, $N\ge 1$, the connected components of $K$. Attach to $K$ a union $E$ of $N$ mutually disjoint smooth embedded Jordan arcs $\gamma_1,\ldots,\gamma_N$ in $M$ meeting $K$ precisely at their two endpoints and such that their intersections with the boundary $bK$ of $K$ are transverse. Then the compact set $S=K\cup E\subset M$ is admissible \cite[Definition 1.12.9]{AlarconForstnericLopez2021}.  Further, 
choose each arc $\gamma_i$ to have one endpoint in $bK_i$ and the other endpoint in $bK_j$ for some $j\neq i$, in such a way that $S$ is connected.  Since this process creates no new loops (because $K$ has one more component than the arcs we are adding), $S$ is holomorphically convex in $M$. 
Note that the first homology group $H_1(S,\Z)$ of $S$ is isomorphic to $\Z^r$ for some integer $r\ge 0$ (see \cite[Lemma 1.12.10]{AlarconForstnericLopez2021}) and the inclusion $K\hookrightarrow S$ induces an isomorphism $H_1(K,\Z)\to H_1(S,\Z)$. Let $\ell'=\ell+r$. By \cite[proof of Proposition 3.3.2]{AlarconForstnericLopez2021}, there are smooth oriented embedded Jordan curves $C_{\ell+1},\ldots, C_{\ell'}$ in $S$ forming a homology basis of $S$ such that $\bigcup_{l=1}^{\ell'} C_l$ is holomorphically convex in $M$ and each $C_l$, $l=1,\ldots,\ell'$, contains a nontrivial arc disjoint from $\bigcup_{i\neq l} C_i$. By Lemma \ref{lem:3.3.1}, we may approximate $f$ uniformly on $S$ by a holomorphic map $f':M\to A$ agreeing with $f$ to any given order on any given finite set in $K$ and satisfying $\int_{C_l}f'\theta=\int_{C_l}f\theta$ for all $l=1,\ldots,\ell'$. In particular, $\int_Cf'\theta=\mathcal F(C)$ for every loop $C$ in a neighbourhood of $S$ by \eqref{eq:loop-F} and the fact that $C_{\ell+1},\ldots, C_{\ell'}$ generate $H_1(S,\Z)\cong H_1(K,\Z)$. Moreover, if the approximation of $f$ by $f'$ on $S$ is sufficiently close, then $f'$ and $f$ are homotopic on a neighbourhood of $S$ and there is a continuous map $f'':M\to A$ homotopic to $f$ such that $f''=f'$ on a neighbourhood of $S$ (and also outside a bigger neighbourhood of $S$ if desired). Therefore, replacing $f$ by $f''$ and $K$ by $S$, we may assume that $K$ is connected, as claimed. 

We now fix Jordan curves $C_{\ell+1},\ldots, C_{\ell'}$ in the connected compact domain $K$ forming a homology basis of $K$, where $\ell'=\ell+\dim(H_1(K,\Z))$, such that $C_1\cup\cdots\cup C_{\ell'}$ is holomorphically convex in $M$ and each $C_l$, $l=1,\ldots,\ell'$, contains a nontrivial arc disjoint from $\bigcup_{i\neq l} C_i$. We proceed by induction on the Euler characteristic $\chi(M\setminus K)$ of $M\setminus K$, which is a nonpositive integer. Assume for the basis of the induction that $\chi(M\setminus K)=0$, so $K$ is a strong deformation retract of $M$ and the inclusion $K\hookrightarrow M$ induces an isomorphism $H_1(K,\Z)\to H_1(M,\Z)$. Applying Lemma \ref{lem:3.3.1} again, we may approximate $f$ uniformly on a neighbourhood of $K$ by a holomorphic map $\tilde f:M\to A$ agreeing with $f$ to any given finite order on any given finite set in $K$, such that $\tilde f(M)$ is not contained in a hyperplane in $\C^n$ and $\int_{C_l}\tilde f\theta=\int_{C_l}f\theta$ for $l=1,\ldots,\ell'$. Since $C_{\ell+1},\ldots, C_{\ell'}$ generate the homology of $M$, this and \eqref{eq:loop-F} ensure that $\int_C\tilde f\theta=\mathcal F(C)$ for every loop $C$ in $M$. Moreover, if the approximation is close enough, then $\tilde f$ is homotopic to $f$ over $K$, and since $K$ is a strong deformation retract of $M$, $\tilde f$ is homotopic to $f$ over all of $M$ as well.

For the inductive step, assume now that the theorem holds whenever $K$ is a connected smoothly bounded compact domain with $-\chi(M\setminus K)=k$ for some integer $k\ge 0$, and suppose that $-\chi(M\setminus K)=k+1$.  Attach to $K$ a smooth embedded Jordan arc $E$ in $M$ meeting $K$ precisely at its two endpoints and such that its intersection with the boundary $bK$ of $K$ is transverse.  Then the compact set $S=K\cup E$ is connected and admissible. Choose $E$ so that $S$ is holomorphically convex in $M$ and $-\chi(M\setminus S)=k$, that is, there is a smooth embedded closed Jordan curve in $S$ which contains $E$ whose homology class belongs to $H_1(S,\Z)$ but not to $H_1(K,\Z)$. 
In view of \eqref{eq:loop-F}, we can use Lemma \ref{lem:paths} in order to deform $f:S\to A$ to a continuous map $f':S\to A$ such that $f'=f$ on a neighbourhood of $K$ and $\int_C f'\theta=\mathcal F(C)$ for every loop $C$ in $S$.  Since $f'$ is homotopic to $f$ on $S$, it admits a continuous extension to $M$ which is homotopic to $f$ on all of $M$ (we can in fact choose $f'$ to equal $f$ outside a neighbourhood of $E$ if desired).  Reasoning as above, we can assume that $f'$ is holomorphic on a neighbourhood of $S$, thereby reducing the proof to the case when $-\chi(M\setminus K)=k$. The induction hypothesis then completes the proof. 
\end{proof}

\begin{remark}
Proposition \ref{p:Oka} remains true with $M$ replaced by an arbitrary open Riemann surface and with the cone $A$ being Oka but not necessarily algebraically elliptic. This follows by a standard recursive application of the arguments in the proof we have given and we leave further details to the reader.
\end{remark}

%
%
\begin{proof}[Proof of Theorem \ref{t:sections}]
(i) $\Rightarrow$ (ii): If $s$ is sufficiently close on a neighbourhood $U$ of $K$ to a regular section $t$ of $\mathcal A$ on $M$, then $s$ and $t$ are homotopic through continuous sections of $\mathcal A \,\vert\, U$.

(ii) $\Rightarrow$ (i):  We proceed in two steps. The first one is given by the following claim.
\begin{claim}\label{cl:no-poles}
If (ii) holds, then $s$ can be uniformly approximated on a neighbourhood of $K$ by regular sections $t$ of $\mathcal A$ on $M$ such that $\int_C t=\mathcal F(C)$ for every loop $C$ in $M$ and $\int_{C_l} t=\int_{C_l} s$ for $l=1,\ldots,\ell$.
\end{claim}
Note that the section $t$ given by this claim satisfies (i) except it may not have a pole at each end of $M$.

\begin{proof}
Let $s_0$ be a regular section of $\mathcal A$ on $M$ that is homotopic to $s$ over a neighbourhood of $K$; it exists by (ii). Then $s_0$ is homotopic over all of $M$ to a continuous section $s'$ of $\mathcal A$ on $M$ that equals $s$ on a smaller neighbourhood of $K$ (see Lemma \ref{lem:extension}).  So by replacing $s$ by $s'$, we can assume that $s$ is continuous on $M$ and $s$ and $s_0$ are homotopic on all of $M$. Fix a holomorphic $1$-form $\theta$ vanishing nowhere on $M$ and write 
\begin{equation}\label{eq:ff0}
	f=s/\theta:M\to A \quad\text{and}\quad f_0=s_0/\theta:M\to A.
\end{equation}
Note that $f$ and $f_0$ are homotopic, $f_0$ is holomorphic, and $f$ is continuous on $M$ and holomorphic on a neighbourhood of $K$. By Proposition \ref{p:Oka}, $f$ can be deformed to a holomorphic map $\tilde f:M\to A$ that approximates $f$ uniformly on a neighbourhood of $K$, such that $\tilde f(M)$ is not contained in a hyperplane in $\C^n$,
\begin{equation}\label{eq:tildef-1}
	\int_C\tilde f\theta=\mathcal F(C)\quad \text{for every loop $C$ in $M$} ,
\end{equation}
and
\begin{equation}\label{eq:tildef-2}
	\int_{C_l}\tilde f\theta=\int_{C_l}f\theta\quad \text{for $l=1,\ldots,\ell$}
\end{equation}
(see \eqref{eq:periods-s}).  In particular, $\tilde f:M\to A$ is homotopic to $f_0$ on $M$ and is nonflat in the sense of \cite[Definition 3.1]{AlarconCastro-Infantes2019APDE} and nondegenerate in the sense of \cite[Definition 2.2]{AlarconForstneric2014IM}. Note that $\tilde f\theta$ is a holomorphic section of $\mathcal A$ on $M$ satisfying the conclusion of the claim, except it may not be regular.

Let $L$ be a connected smoothly bounded compact domain in $M=\overline M\setminus\{p_1,\ldots,p_m\}$ with $K\subset\mathring L$ and $\overline M\setminus \overset\circ L = D_1\cup\cdots\cup D_m$, where $D_1,\ldots, D_m$ are mutually disjoint smoothly bounded closed discs centred at $p_1,\ldots, p_m$, respectively. Recall that the first homology group $H_1(M,\Z)$ is isomorphic to $\Z^r$ for some $r\ge 0$ and that the inclusion $L\hookrightarrow M$ induces an isomorphism $H_1(L,\Z)\to H_1(M,\Z)$. Write $\ell'=\ell+r$. By \cite[proof of Proposition 3.3.2]{AlarconForstnericLopez2021}, there are smooth oriented embedded Jordan curves $C_{\ell+1},\ldots, C_{\ell'}$ in $L$ forming a homology basis of $L$ such that 
\[
	\Gamma=C_1\cup\cdots\cup C_{\ell'}
\]
is holomorphically convex in $M$ and each $C_l$, $l=1,\ldots,\ell'$, contains a nontrivial arc disjoint from $\bigcup_{i\neq l} C_i$.
Consider the {\em period-interpolation map} 
\[
	\mathscr P=(\mathscr P_1,\ldots,\mathscr P_{\ell'}): \mathscr C(\Gamma,\C^n)\to(\C^n)^{\ell'}=\C^{\ell' n}
\]
 defined by
\begin{equation}\label{eq:p-map}
	\mathscr P_l(h)=\int_{C_l}h\theta\in\C^n,\quad h\in \mathscr C(\Gamma,\C^n),\quad l=1,\ldots,\ell'.
\end{equation}
Arguing as in \cite[proof of Lemma 3.3.1, Step 2]{AlarconForstnericLopez2021} but with the punctured null quadric replaced by the cone $A$ (see also \cite[Lemma 4.2]{AlarconCastro-Infantes2019APDE} and \cite[Lemma 5.1]{AlarconForstneric2014IM}), we can find an open ball $V$ centred at the origin $0$ in some $\C^k$ and a holomorphic {\em $\mathscr P$-dominating spray} $\phi: V\times M\to A$ with {\em core}\, $\tilde f$, that is, a holomorphic map such that $\phi(0,\cdot)=\tilde f$ and
\[
	\frac{\partial}{\partial \zeta}\Big|_{\zeta=0}\mathscr P\big(\phi(\zeta,\cdot)\big): (\C^n)^{\ell'}\to (\C^n)^{\ell'}\quad \text{is an isomorphism}.
\]
(Note that \cite[Lemma 4.2]{AlarconCastro-Infantes2019APDE} and \cite[Lemma 5.1]{AlarconForstneric2014IM} deal with the case when $M$ is a compact bordered Riemann surface, but the same proof applies when $M$ is an arbitrary open Riemann surface since we only need to control the periods on the finitely many arcs and curves $C_1,\ldots, C_{\ell'}$.)  We may assume that the spray $\phi$ extends continuously to a map $\phi:\C^k\times M\to A$.  Since $A$ satisfies the basic Oka property with approximation and interpolation (see \cite[Corollary 5.4.5]{Forstneric2017E}), $\phi$ may be deformed to a holomorphic spray $\sigma:\C^k\times M\to A$ such that $\sigma(0,\cdot)=\tilde f$, with sufficiently close approximation on a neighbourhood of $\{0\}\times L\supset \{0\}\times \Gamma$ to ensure that $\sigma$ is still $\mathscr P$-dominating:
\begin{equation}\label{eq:P-sigma}
	\frac{\partial}{\partial \zeta}\Big|_{\zeta=0}\mathscr P\big(\sigma(\zeta,\cdot)\big): (\C^n)^{\ell'}\to (\C^n)^{\ell'}\quad \text{is an isomorphism}.
\end{equation}
Using $\theta$, we turn $\sigma$ into a holomorphic spray $\sigma\theta$ of sections of $\mathcal A$ on $M$ parametrised by $\C^k$ such that $\sigma(0,\cdot)\theta=\tilde f\theta$.  We may view the spray $\sigma\theta$ as a holomorphic section of the pullback of $\mathcal A$ by the projection $\C^k\times M\to M$.  By Forstneri\v c's algebraic approximation theorem \cite[Theorem 3.1]{Forstneric2006AJM} (see also \cite[Theorem 6.15.3]{Forstneric2017E}), since $A$ is algebraically elliptic and $\tilde f\theta$ is homotopic to the regular section $f_0\theta=s_0$ of $\mathcal A$ on $M$ (see \eqref{eq:ff0}), we may uniformly approximate the spray $\sigma\theta$ over $W\times L$ by a regular spray $\tau$ of sections of $\mathcal A$ on $M$ parametrised by $\C^k$, where $W$ is a closed ball centred at the origin in $\C^k$. By \eqref{eq:P-sigma}, if the approximation is close enough then near the origin is a parameter $\zeta_0\in\C^k$ such that 
\begin{equation}\label{eq:periods-t}
	\Pscr\big(\tau(\zeta_0)/\theta\big)=\Pscr\big(\sigma(0,\cdot)\big)=\Pscr\big(\tilde f).
\end{equation}
Then $t=\tau(\zeta_0)$ is a regular section of $\mathcal A$ on $M$ which is close to $\tilde f\theta =\sigma(0,\cdot)\theta$ on $L$ and hence to $s=f\theta$ on $K$; see \eqref{eq:ff0}. It follows from \eqref{eq:p-map} and \eqref{eq:periods-t} that 
\begin{equation}\label{eq:tildef-3}
	\int_{C_l} t=\int_{C_l}\tilde f\theta\quad \text{for }l=1,\ldots,\ell'.
\end{equation}
This, \eqref{eq:ff0}, and \eqref{eq:tildef-2} show that
$\int_{C_l} t=\int_{C_l} f\theta=\int_{C_l} s$ for $l=1,\ldots,\ell$.
Finally, since the curves $C_{\ell+1},\ldots, C_{\ell'}$ form a basis of $H_1(M,\Z)$,  \eqref{eq:tildef-1} and \eqref{eq:tildef-3} imply that 
$\int_C t =\int_C\tilde f\theta=\mathcal F(C)$ for every loop $C$ in $M$.
\end{proof}

To complete the proof that (ii) $\Rightarrow$ (i), it remains to see that we can find a regular section $t$ as in Claim \ref{cl:no-poles} with an effective pole at each end of $M$. For this, we adapt the argument in the proof of Proposition \ref{pro:large} to our current more general framework.  By Claim \ref{cl:no-poles}, we may assume that $s$ extends to a regular section of $\mathcal A$ on $M$.
Let $L\Supset K$ be a smoothly bounded compact domain in $M=\overline M\setminus\{p_1,\ldots,p_m\}$ such that $\overline M\setminus \overset\circ L = D_1\cup\cdots\cup D_m$, where $D_1,\ldots, D_m$ are mutually disjoint smoothly bounded closed discs centred at $p_1,\ldots, p_m$, respectively.  In particular, $bL = bD_1 \cup\cdots\cup bD_m$. Fix a holomorphic $1$-form $\vartheta$ vanishing nowhere on a neighbourhood of $D=D_1\cup\cdots\cup D_m\supset bL$. Also choose for each $i=1,\ldots,m$ a smoothly bounded closed disc $\Delta_i\subset \mathring D_i\setminus\{p_i\}$, set $\Delta=\Delta_1\cup\cdots\cup\Delta_m$, and consider a holomorphic section $\tilde s$ of $\mathcal A$ on a neighbourhood of $L\cup\Delta$ such that $\tilde s=s$ on a neighbourhood of $L$ and the map
\begin{equation}\label{eq:g}
	g=(g_1,\ldots,g_n)=\frac{\tilde s}{\vartheta}:\Delta\cup bL\to A
\end{equation}
satisfies
\begin{equation}\label{eq:g1}
	 \min\{|g_k(p)|: p\in \Delta\}>\max\{|g_k(p)|: p\in bL\}\quad \text{for }k=1,\ldots,n.
\end{equation}
(Simply define $\tilde s$ on $\Delta$ as 
$\tilde s=g\vartheta$ where $g=(g_1,\ldots,g_n):\Delta\to A$ is a holomorphic map with $\min\{|g_k(p)|: p\in \Delta\}>\max\{|(s/\vartheta)(p)| : p\in bL\}$ for $k=1,\ldots,n$; note that $\Delta$ is compact, $L\cap\Delta=\varnothing$, and $g=s/\vartheta$ on $bL$.  We can even choose $g$ to be constant on $\Delta$.) Since $\Delta$ is simply connected and $\tilde s=s$ on $L$, $\tilde s$ and the regular section $s$ of $\mathcal A$ on $M$ are homotopic over a neighbourhood of $L\cup\Delta$, so condition (ii) is satisfied by the holomorphic section $\tilde s$ of $\mathcal A$ on $L\cup\Delta$.  In view of Claim \ref{cl:no-poles}, $\tilde s$ can be uniformly approximated on a neighbourhood of the holomorphically convex compact set $L\cup\Delta\subset M$ by a regular section $t$ of $\mathcal A$ on $M$ such that $\int_C t=\mathcal F(C)$ for every loop $C$ in $M$ and $\int_{C_l} t=\int_{C_l} s$ for $l=1,\ldots,\ell$.  Take such a section $t$ so close to $\tilde s$ on $L\cup\Delta$ that the holomorphic map $h=(h_1,\ldots,h_n)=t/\vartheta:D\setminus\{p_1,\ldots,p_m\}\to A$ satisfies
\begin{equation}\label{eq:h1}
	\min\Big\{|h_k(p)|: p\in \Delta=\bigcup_{i=1}^m\Delta_i\Big\}>\max\Big\{|h_k(p)|: p\in bL=\bigcup_{i=1}^m bD_i\Big\}; 
\end{equation}
for $k=1,\ldots,n$;
see \eqref{eq:g} and \eqref{eq:g1}. For each $i=1,\ldots,m$ and $k=1,\ldots,n$, since $t=(t_1,\ldots,t_n)$ is regular on $M$ and $\vartheta$ is holomorphic at $p_i$, either $h_k=t_k/\vartheta$ has an effective pole at $p_i$ or $h_k$ is holomorphic at $p_i$. To complete the proof it now suffices to check that $h_k$ has a pole at $p_i$ for all $i=1,\ldots,m$ and $k=1,\ldots,n$. Indeed, since $\vartheta$ is holomorphic and nowhere vanishing on $D\supset\{p_1,\ldots, p_m\}$, this would imply that $t_k|_{D\setminus\{p_1,\ldots,p_m\}}=h_k\vartheta$ has a pole at $p_i$, $i=1,\ldots,m$, and so $t_k$ would have a pole at each end $p_i$ of $M$ as well. So fix $i\in\{1,\ldots,m\}$ and $k=1,\ldots,n$, and suppose that $h_k$ is holomorphic at $p_i$. Then $h_k$ is holomorphic on the smoothly bounded closed disc $D_i$, violating the maximum modulus principle, since $\Delta_i\subset\mathring D_i$ and $\min\{|h_k(p)|: p\in \Delta_i\}>\max\{|h_k(p)|: p\in bD_i\}$ by \eqref{eq:h1}. Therefore, each $h_k$ has an effective pole at each end $p_i$ of $M$. 
\end{proof}

Here is a more precise version of Theorem \ref{t:cohomology}.
%
%
\begin{theorem}\label{t:cohomology-2}
Let $A$, $M$, and $\mathcal A$ be as in Theorem \ref{t:main-theorem} and assume that $t_0$ is a regular section of $\mathcal A$ on $M$. Then for any group homomorphism $\mathcal F:H_1(M,\Z)\to\C^n$ there is a regular section $t=(t_1,\ldots,t_n)$ of $\mathcal A$ on $M$, homotopic to $t_0$, such that all the components $1$-forms $t_1,\ldots,t_n$ have an effective pole at each end of $M$ and $\int_C t=\mathcal F(C)$ for every loop $C$ in $M$.
\end{theorem}
\begin{proof}
Let $\theta$ be a holomorphic $1$-form vanishing nowhere on $M$ and consider the holomorphic map $f_0=t_0/\theta:M\to A$.
Recall that $H_1(M,\Z)\cong\Z^r$ for some integer $r\ge 0$. Fix a point $p_0\in M$ and let $C_1,\ldots,C_r$ be a family of smooth oriented embedded Jordan curves forming a basis of $H_1(M,\Z)$ such that $C_i\cap C_j=\{p_0\}$ for all $i\neq j$ and the compact set $\Gamma=C_1\cup\cdots\cup C_r$ is holomorphically convex in $M$, and hence a strong deformation retract of $M$. By Lemma \ref{lem:paths}, we can deform $f_0|_\Gamma$ outside a neighbourhood of $p_0$ to a continuous map $f:\Gamma\to A$ with $\int_{C_i}f\theta=\mathcal F(C_i)$ for $i=1,\ldots,r$. Then Lemma \ref{lem:3.3.1} enables us to approximate $f$ uniformly on $\Gamma$ by a holomorphic map $f':M\to A$ with $\int_C f'\theta=\mathcal F(C)$ for every loop $C$ in $M$. If the approximation is close enough, then $f'$ is homotopic to $f$ on $\Gamma$, so $f'$ and $f_0$ are homotopic on $M$. Set $s=f'\theta$, which is a holomorphic section of $\mathcal A$ on $M$, homotopic to the given regular section $t_0=f_0\theta$. Let $K\subset M$ be a smoothly bounded compact domain that is a strong deformation retract of $M$. By Theorem \ref{t:sections}, $s$ can be approximated uniformly on $K$ by a regular section $t=(t_1,\ldots,t_n)$ of $\mathcal A$ on $M$, such that $t_k$ has an effective pole at each end of $M$ for all $k=1,\ldots,n$, and $\int_C t=\mathcal F(C)$ for every loop $C$ in $M$. If the approximation on $K$ is close enough, then $t$ is homotopic to $s$, hence to $t_0$.
\end{proof}


\section{Regular directed immersions}
\label{sec:proofs}

\noindent
In this section we first derive Theorem \ref{t:main-theorem} on regular immersions directed by algebraically elliptic cones from Theorem \ref{t:sections}.  Then we prove Corollaries \ref{c:second-corollary} and~\ref{c:main-corollary}.
%
%
\begin{proof}[Proof of Theorem \ref{t:main-theorem}]
By Corollary \ref{co:large-very-good}, we may assume that $n\ge 3$ and $A\neq \C_*^n$. We proceed in three steps.

(ii) $\Rightarrow$ (i) and (iv) $\Rightarrow$ (iii) are obvious.

(i) $\Rightarrow$ (iii) and (ii) $\Rightarrow$ (iv):  
If $u$ is near a regular $A$-immersion $v:M\to\C^n$ on a neighbourhood of $K$, then $du$ is near $dv$ on a smaller neighbourhood $U$ of $K$.  If $du$ and $dv$ are sufficiently close on $U$, then they are homotopic through continuous sections of $\mathcal A \,\vert\, U$.  Clearly, $dv$ is exact, and since $v$ is a regular $A$-immersion, $dv$ is a regular section of $\mathcal A$. If in addition $v$ is proper, then $v$, and hence $dv$, has an effective pole at each end of $M$.

(iii) $\Rightarrow$ (ii): 
Assume without loss of generality that $K$ is a smoothly bounded compact domain, possibly disconnected, and let $\Lambda=\{q_1,\ldots,q_\ell\} \subset K$ be a finite set. We may assume that $\Lambda\subset\mathring K$. 

Set $s=du$, so $s$ is an exact holomorphic section of $\mathcal A$ on a neighbourhood of $K$:
\begin{equation}\label{eq:s-exact-K}
	\int_C s=0\quad \text{for every loop $C$ in a neighbourhood of $K$}.
\end{equation}

Assume first that $K$ is connected. Fix a point $p_0\in \mathring K\setminus \Lambda$. It is clear that 
\begin{equation}\label{eq:u(p)}
	u(p)=u(p_0)+\int_{p_0}^p du\quad \text{for any point $p\in K$},
\end{equation}
where the integral is computed over any oriented arc in $K$ connecting $p_0$ and $p$.
Let $C_1,\ldots,C_\ell$ be smooth oriented embedded arcs in $\mathring K$ such that $C_i\cap C_j=\{p_0\}$ for all $i,j\in\{1,\ldots,\ell\}$, $i\neq j$, and the endpoints of $C_l$ are $p_0$ and $q_l$, $l=1,\ldots,\ell$. By virtue of (iii), we may use Theorem \ref{t:sections} to approximate $s$ uniformly on a neighbourhood of $K$ by a regular section $t=(t_1,\ldots,t_n)$ of $\mathcal A$ on $M$, such that $t_k$ has an effective pole at each end of $M$ for all $k=1,\ldots,n$,
\begin{equation}\label{eq:t-exact}
	\int_C t=0\quad \text{for every loop $C$ in $M$}
\end{equation}
(see \eqref{eq:s-exact-K} and take $\mathcal F=0$ everywhere on $H_1(M,\Z)$) and
\begin{equation}\label{eq:t-interpolates}
	\int_{C_l} t=\int_{C_l} s\quad \text{for $l=1,\ldots,\ell$}.
\end{equation}
By \eqref{eq:t-exact}, $t$ is exact. It is also regular, and hence the map $v=(v_1,\ldots,v_n):M\to\C^n$ given by
\[
	v(p)=u(p_0)+\int_{p_0}^p t,\quad p\in M,
\]
is a regular $A$-immersion. It is clear that $v$ is close to $u$ uniformly on $K$, since $K$ is connected and $t$ is close to $s=du$ uniformly on $K$; see \eqref{eq:u(p)}. Moreover, since $t_k$ has an effective pole at each end of $M$, so does $v_k$, so $v_k:M\to\C$ is proper for all $k=1,\ldots,n$. In particular, $v:M\to\C^n$ is proper as well. Finally, by \eqref{eq:t-interpolates}, $v=u$ on $\Lambda$. This proves (ii) in the case when $K$ is connected.

Assume now that the compact domain $K$ is disconnected.  Denote by $K_0,\ldots,K_N$, $N\ge 1$, the connected components of $K$.  As in the first part of the proof of Proposition \ref{p:Oka}, attach to $K$ the union $E=\gamma_1\cup\cdots\cup\gamma_N$ of $N$ mutually disjoint smooth embedded Jordan arcs in $M$, such that each arc $\gamma_i$ 
has one endpoint in $bK_i$ and the other endpoint in $bK_j$ for some $j\neq i$, so as to obtain a connected holomorphically convex compact admissible set $S=K\cup E\subset M$.  Since this procedure does not create new loops, the inclusion $K\hookrightarrow S$ induces an isomorphism $H_1(K,\Z)\to H_1(S,\Z)$. In this case, condition \eqref{eq:s-exact-K} implies that \eqref{eq:u(p)} holds for any pair of points $p$ and $p_0$ in the same connected component of $K$. In view of (iii), there is a regular section $t$ of $\mathcal A$ on $M$ which is homotopic to $s$ over a neighbourhood of $K$. Fix a holomorphic $1$-form $\theta$ vanishing nowhere on $M$ and consider the holomorphic maps
\begin{equation}\label{eq:fg}
	f=s/\theta:K\to A \quad \text{and}\quad g=t/\theta:M\to A
\end{equation}
which are homotopic over a neighbourhood of $K$. Extend $f$ to a continuous map $f:S\to A$ which is holomorphic on a neighbourhood of $K$ such that $f$ and $g$ are homotopic on $S$.  
Fix orientations for the arcs $\gamma_1,\ldots,\gamma_N$ and choose points $\mu_1,\ldots,\mu_N\in\C^n$ which will be specified later. Reasoning as in the final part of the proof of 
Proposition \ref{p:Oka}, we can now use 
Lemma \ref{lem:paths} in order to deform $f:S\to A$ to a continuous map $f':S\to A$ such that $f'=f$ on a neighbourhood of $K$ and 
\begin{equation}\label{eq:f'}
	\int_{\gamma_i} f'\theta=\mu_i\quad \text{for $i=1,\ldots,N$}.
\end{equation} 
Then $f'$ and $g$ are homotopic over $S$. Moreover, the continuous $1$-form $f'\theta$ on $S$ is exact, since so is $f\theta=s=du$ on $K$ (see \eqref{eq:s-exact-K}), $f'=f$ on a neighbourhood of $K$, and there is no closed curve in $S$ whose homology class belongs to $H_1(S,\Z)$ but not to $H_1(K,\Z)$. Thus, fixing a point $p_0\in \mathring K_0\setminus \Lambda$, $f'\theta$ integrates to a continuous map $u':S\to\C^n$ given by
\begin{equation}\label{eq:u'}
	u'(p)=u(p_0)+\int_{p_0}^p f'\theta,\quad p\in S.
\end{equation}
By \eqref{eq:u(p)}, \eqref{eq:u'}, and since $f'=f$ on $K$, we have that $u'=u$ in a neighbourhood of $K_0$, while $u'-u$ is constant on a neighbourhood of $K_i$ for $i=1,\ldots,N$. In fact, suitably choosing the constants $\mu_i$ in \eqref{eq:f'}, we can ensure that 
\begin{equation}\label{eq:u'=u}
	\text{$u'=u$ on a neighbourhood of $K=K_0\cup K_1\cup\cdots\cup K_N$}.
\end{equation} 
Choose smooth oriented embedded Jordan arcs $C_1,\ldots,C_\ell$ in $S$ such that $C_1\cup\cdots\cup C_\ell$ is holomorphically convex in $M$ and each arc $C_l$ has initial point $p_0$ and final point $q_l\in\Lambda$, $l=1,\ldots,\ell$. It follows from \eqref{eq:u'} and \eqref{eq:u'=u} that 
\begin{equation}\label{eq:Clf'}
	\int_{C_l}f'\theta=u(q_l)-u(p_0)\quad \text{for } l=1,\ldots,\ell.
\end{equation}
Set $\ell'=\ell+\dim(H_1(S,\Z))$. By \cite[proof of Proposition 3.3.2]{AlarconForstnericLopez2021}, there are smooth oriented embedded Jordan arcs $C_{\ell+1},\ldots, C_{\ell'}$ in $S$ forming a homology basis of $S$ such that $\bigcup_{l=1}^{\ell'} C_l$ is holomorphically convex in $M$ and each arc $C_l$, $l=1,\ldots,\ell'$, contains a nontrivial arc disjoint from $\bigcup_{i\neq l} C_i$. By Lemma \ref{lem:3.3.1}, we can approximate $f'$ uniformly on $S$ by a holomorphic map $F:M\to A$ such that
\begin{equation}\label{eq:periods-F}
	\int_{C_l} F\theta=\int_{C_l} f'\theta\quad \text{for } l=1,\ldots,\ell'.
\end{equation}
Since $f'\theta$ is exact on $S$ and $C_{\ell+1},\ldots, C_{\ell'}$ generate $H_1(S,\Z)$, it follows that $F\theta$ is exact on a connected neighbourhood $U$ of $S$, and hence it integrates to an $A$-immersion $v:U\to\C^n$  given by
\begin{equation}\label{eq:vuF}
	v(p)=u(p_0)+\int_{p_0}^pF\theta,\quad p\in U.
\end{equation}
Since $F$ is close to $f'$ on the connected compact set $S$, \eqref{eq:u'} and \eqref{eq:vuF} ensure that $v$ is close to $u'$ on a neighbourhood of $S$, and hence to $u$ on a neighbourhood of $K$; see \eqref{eq:u'=u}. Moreover, by \eqref{eq:Clf'} and \eqref{eq:periods-F}, $v=u$ on $\Lambda$. Finally, if the approximation of $f'$ by $F$ is close enough, then $F$ is homotopic to $f'$, and hence to $g$, over $S$.  Thus $F\theta$ is homotopic to the regular section $t=g\theta$ of $\mathcal A$ on $M$ over a neighbourhood of $S$; see \eqref{eq:fg}. This reduces the proof to the already settled case when $K$ is connected. 
\end{proof}

%
%
\begin{proof}[Proof of Corollary \ref{c:second-corollary}]
The implications {\rm (ii)} $\Rightarrow$  {\rm (iv)} $\Rightarrow$ {\rm (i)} $\Rightarrow$ {\rm (iii)} are obvious.  To show that {\rm (iii)} implies {\rm (ii)}, we simply apply Theorem \ref{t:main-theorem} to any holomorphic $A$-immersion $u:K\to\C^n$, where $K\subset M$ is a simply connected, smoothly bounded compact neighbourhood of $\Lambda$, such that $u$ extends the given map $\Lambda\to\C^n$. Such an $A$-immersion trivially exists; note that a translate of an $A$-immersion is still an $A$-immersion.  Since $K$ is simply connected, there is only one homotopy class of continuous sections of $\mathcal A\,\vert\, U$ on a suitable neighbourhood of $K$, hence $du$ is homotopic to the restriction of the regular section of $\mathcal A$ on $M$ given by (iii) and the theorem applies. 
\end{proof}

%
%
\begin{proof}[Proof of Corollary \ref{c:main-corollary}]
%
(ii) $\Rightarrow$ (i) and (iv) $\Rightarrow$ (iii) are obvious, while (i) $\Rightarrow$ (iii) and (ii) $\Rightarrow$ (iv) are seen as in the proof of Corollary \ref{co:large-very-good}.

\noindent
(iii) $\Rightarrow$ (ii):  Let $u:K\to\C^n$ be a holomorphic $A$-immersion as in (i).  By \cite[Theorem 1.3]{AlarconCastro-Infantes2019APDE}, $u$ can be uniformly approximated on $K$ by holomorphic $A$-immersions $v:M\to\C^n$ agreeing with $u$ on any given finite subset of $K$.  Assuming (iii), apply Theorem \ref{t:main-theorem} to approximate and interpolate $v$ and hence $u$ on $K$ to obtain (ii).
\end{proof}
%
%
\begin{remark}\label{rem:vj}
An inspection of the proof of Theorem \ref{t:main-theorem} shows, assuming the equivalent conditions hold, that the approximation in (ii) can be done by $A$-immersions $M\to\C^n$ all of whose component functions are proper regular functions on $M$. Likewise, the exact regular section of $\mathcal A$ on $M$ in (iv) can be chosen so that all its component $1$-forms have an effective pole at each end of $M$. The same applies to conditions (ii) and (iv) in Corollaries \ref{c:second-corollary} and \ref{c:main-corollary}.
\end{remark}


\section{Directed harmonic maps}\label{sec:harmonic}

\noindent
Our method of proof yields a result analogous to Theorem \ref{t:main-theorem} for directed harmonic maps into $\R^n$. If $M$ is an open Riemann surface and $A$ and $\mathcal A$ are as in Theorem \ref{t:main-theorem}, a harmonic map $\phi:M\to\R^n$ is said to be {\em directed by} $A$, or to be an {\em $A$-map}, if the $(1,0)$-part $\partial \phi$ of its exterior derivative $d\phi$ is a holomorphic section of $\mathcal A$ on $M$. If $M$ is a smooth affine curve, we say that $\phi$ is regular if $\partial\phi$ is.  The {\em flux map} (or just the {\em flux}) of $\phi$ is the group homomorphism ${\rm Flux}_\phi:H_1(M,\Z)\to\R^n$ defined by
\[
	{\rm Flux}_\phi(C)=\int_Cd^c\phi \quad \text{for every loop $C$ in $M$},
\]
where $d^c\phi=\imath(\overline\partial\phi-\partial\phi)$ is the conjugate differential of $\phi$.
A harmonic $A$-map $\phi:M\to\R^n$ is the real part of a holomorphic $A$-immersion $M\to\C^n$ if and only if  ${\rm Flux}_\phi(C)=0$ for every loop $C$ in $M$ or, equivalently, $d^c\phi$ is exact. 

We record the following third main result of this paper.  
We say that a harmonic map $\phi:M\to\R^n$ is regular if its $(1,0)$-differential $\partial \phi$ is a regular $1$-form on $M$.
%
%
\begin{theorem}   \label{t:second-theorem}
Let $A$, $M$, and $\mathcal A$ be as in Theorem \ref{t:main-theorem}. Let $K$ be a holomorphically convex compact subset of $M$, $\phi:K\to\R^n$ be a harmonic $A$-map on a neighbourhood of $K$, and $\mathcal F:H_1(M,\Z)\to \R^n$ be a group homomorphism with $\mathcal F(C)={\rm Flux}_\phi(C)$ for every loop $C$ in a neighbourhood of $K$.  Then the following are equivalent.
\begin{enumerate}[{\rm (i)}]
\item  $\phi$ can be uniformly approximated on a neighbourhood of $K$ by regular harmonic $A$-maps $M\to\R^n$.
\item  $\phi$ can be uniformly approximated on a neighbourhood of $K$ by regular harmonic $A$-maps $\psi:M\to\R^n$ agreeing with $\phi$ on any given finite set in $K$ and with flux map $\mathcal F$, such that $\partial\psi$ has an effective pole at each end of $M$.
\item  There is a neighbourhood $U$ of $K$ such that the homotopy class of continuous sections of $\mathcal A\,\vert\, U$ that contains $\partial\phi$ also contains the restriction of a regular section of $\mathcal A$ on $M$.
\item  There is a neighbourhood $U$ of $K$ such that the homotopy class of continuous sections of $\mathcal A\,\vert\, U$ that contains $\partial\phi$ also contains the restriction of a regular section $t$ of $\mathcal A$ on $M$ with an effective pole at each end of $M$ and representing any given class in the cohomology group $H^1(M,\C^n)$.
\end{enumerate}
\end{theorem}
Note that (iii) $\Leftrightarrow$ (iv) is guaranteed by Theorem \ref{t:cohomology}.  As in Remark \ref{rem:vj}, the sections $\partial\psi$ in (ii) and $t$ in (iv) can be chosen such that all their component $1$-forms have an effective pole at each end of $M$.
The proof of Theorem \ref{t:second-theorem} is very similar to that of Theorem \ref{t:main-theorem}, relying on the results in Section \ref{sec:sections}, the only difference being that now we can ignore the imaginary periods of the approximating sections. 

In general, harmonic $A$-maps need not be immersions; by a harmonic $A$-{\em immersion} we mean a harmonic $A$-map which is an immersion. Given a harmonic map $\phi=(\phi_1,\ldots,\phi_n):M\to\R^n$ on an open Riemann surface $M$, the holomorphic 2-form $\mathfrak H_\phi=\sum_{i=1}^n(\partial\phi_i)^2$ on $M$ is called the {\em Hopf differential} of $\phi$. It turns out that $\phi$ is an immersion if and only if $|\mathfrak H_\phi|<|\partial\phi|^2=\sum_{i=1}^n|\partial\phi_i|^2$ everywhere on $M$; see \cite[Lemma 2.4]{AlarconLopez2013TAMS}. This happens if and only if $\partial\phi/\theta$ takes its values in $\C^n\setminus (\c\cdot\R^n)$ for any nowhere-vanishing holomorphic 1-form $\theta$ on $M$.  In particular, if $A$ and $\mathcal A$ are as in Theorem \ref{t:second-theorem} with $A\cap\R^n=\varnothing$ (hence $n\ge 3$) and $\partial\phi$ is a section of $\mathcal A$, then $\phi$ is a harmonic $A$-immersion. So, the following holds.
\begin{remark}
Assuming in Theorem \ref{t:second-theorem} that $A\cap\R^n=\varnothing$, we have that the given map $\phi$ is a harmonic $A$-immersion. If, in addition, the equivalent conditions (i)--(iv) hold, then the approximating maps $\psi$ in (ii) are harmonic $A$-immersions as well.
\end{remark}

The punctured null quadric $A\subset\C_*^n$ in \eqref{eq:null}, directing holomorphic null curves in $\C^n$ and minimal surfaces in $\R^n$, is very good by Proposition \ref{pro:null-very-good}. It also satisfies $A\cap\R^n=\varnothing$. Thus, by a direct application of Theorem \ref{t:second-theorem}, we recover the following result from \cite{AlarconLopez2022APDE} (see \cite{Lopez2014TAMS, AlarconCastro-InfantesLopez2019CVPDE} for the case $n=3$); see Section \ref{ss:null} for background.
%
%
\begin{corollary}\label{c:CMI}
Let $M$ be a smooth affine curve, $K\subset M$ be a holomorphically convex compact set, $\phi:K\to\R^n$, $n\ge 3$, be a conformal minimal immersion on a neighbourhood of $K$, and $\mathcal F:H_1(M,\Z)\to\R^n$ be a group homomorphism such that $\mathcal F(C)={\rm Flux}_\phi(C)$ for every loop $C$ in a neighbourhood of $K$. Then $\phi$ can be uniformly approximated on a neighbourhood of $K$ by complete conformal minimal immersions $\psi:M\to\R^n$ of finite total curvature agreeing with $\phi$ on any given finite set in $K$ and satisfying ${\rm Flux}_\psi=\mathcal F$.
\end{corollary}
As above, the approximating conformal minimal immersions $\psi=(\psi_1,\ldots,\psi_n):M\to\R^n$ in the corollary can be chosen such that $\partial\psi_k$ has an effective pole at each end of $M$ for all $k=1,\ldots,n$.


\subsection*{Acknowledgements}
A.~Alarc\'on was partially supported by the State Research Agency (AEI) 
via the grants no.\ PID2020-117868GB-I00 and PID2023-150727NB-I00, and the ``Maria de Maeztu'' Excellence Unit IMAG, reference CEX2020-001105-M, funded by MCIN/AEI/10.13039/501100011033 and ERDF/EU.  The authors thank three anonymous referees for helpful comments.




\vspace*{-1mm}
\medskip
\noindent Antonio Alarc\'{o}n

\noindent Departamento de Geometr\'{\i}a y Topolog\'{\i}a e Instituto de Matem\'aticas (IMAG), Universidad de Granada, Campus de Fuentenueva s/n, E--18071 Granada, Spain

\noindent  e-mail: {\tt alarcon@ugr.es}

\bigskip
\noindent Finnur L\'arusson

\noindent Discipline of Mathematical Sciences, University of Adelaide, Adelaide SA 5005, Australia

\noindent  e-mail: {\tt finnur.larusson@adelaide.edu.au}

\end{document}